\documentclass{amsart}

\usepackage{amsmath,amssymb,amsthm}
\usepackage{pinlabel}

\newtheorem{prop}{Proposition}[section]
\newtheorem{thm}[prop]{Theorem}
\newtheorem{lem}[prop]{Lemma}
\newtheorem{cor}[prop]{Corollary}

\theoremstyle{definition}

\newtheorem{ex}[prop]{Example}
\newtheorem{rem}[prop]{Remark}

\newtheorem*{ack}{Acknowledgements}

\newcommand{\tb}{\mathtt{tb}}
\newcommand{\lk}{\mathtt{lk}}
\newcommand{\vlk}{\underline{\lk}}
\newcommand{\rot}{\mathtt{rot}}
\newcommand{\vrot}{\underline{\rot}}
\newcommand{\ttt}{\mathtt{t}}

\newcommand{\rmd}{\mathrm{d}}

\newcommand{\bfx}{\mathbf{x}}

\newcommand{\N}{\mathbb{N}}
\newcommand{\Q}{\mathbb{Q}}

\newcommand{\Z}{\mathbb{Z}}

\newcommand{\LL}{\mathbb{L}}

\newcommand{\xist}{\xi_{\mathrm{st}}}
\newcommand{\xiot}{\xi_{\mathrm{ot}}}

\begin{document}

\author[H.~Geiges]{Hansj\"org Geiges}
\address{Mathematisches Institut, Universit\"at zu K\"oln,
Weyertal 86--90, 50931 K\"oln, Germany}
\email{geiges@math.uni-koeln.de}

\author[S.~Onaran]{Sinem Onaran}
\address{Department of Mathematics, Hacettepe University,
06800 Beytepe-Ankara, Turkey}
\email{sonaran@hacettepe.edu.tr}

\title{Legendrian lens space surgeries}

\date{}

\begin{abstract}
We show that every tight contact structure on any of the
lens spaces $L(ns^2-s+1,s^2)$ with $n\geq 2$, $s\geq 1$,
can be obtained by a single Legendrian surgery along
a suitable Legendrian realisation of the negative torus
knot $T(s,-(sn-1))$ in the tight or an overtwisted contact
structure on the $3$-sphere.
\end{abstract}


\maketitle


\section{Introduction}
A knot $K$ in the $3$-sphere $S^3$ is said to admit
a \emph{lens space surgery} if, for some rational number $r$,
the $3$-manifold obtained by Dehn surgery along $K$ with surgery
coefficient $r$ is a lens space. In \cite{mose71} L.~Moser showed
that all torus knots admit lens space surgeries. More precisely,
$-(ab\pm 1)$-surgery along the negative torus knot $T(a,-b)$
results in the lens space $L(ab\pm 1,a^2)$, cf.~\cite{taya12};
for positive torus knots one takes the mirror of the knot and the
surgery coefficient of opposite sign, resulting in a negatively
oriented lens space.
Contrary to what was conjectured by Moser,
there are surgeries along other knots that produce lens spaces.
The first example was due to J.~Bailey and D.~Rolfsen~\cite{baro77},
who constructed the lens space $L(23,7)$ by integral surgery along
an iterated cable knot.

The question which knots admit lens space surgeries is
still open and the subject of much current research.
The fundamental result in this area is
due to Culler--Gordon--Luecke--Shalen~\cite{cgls87}, proved as a corollary
of their cyclic surgery theorem: if $K$ is not a torus knot, then at most
two surgery coefficients, which must be successive integers, can correspond
to a lens space surgery. For more recent work, relating this
question to Floer theory, see \cite{hedd11,kmos07,ozsz05,rasm07},
for instance.

The converse question, which lens spaces can be obtained by a single surgery
on the $3$-sphere, is of course trivial in the topological setting:
the lens space $L(p,q)$, as an oriented manifold,
is the result of a $(-p/q)$-surgery along the unknot.

In the present note we consider this converse question for contact
manifolds: which
tight contact structures on a given lens space can be obtained
by a single contact $(-1)$-surgery (also known as Legendrian surgery)
along a Legendrian knot in $S^3$ with some contact structure?
Here the topological restrictions become relevant in the search for
contact structures on lens spaces that cannot be constructed in this way.

J. Rasmussen proved in \cite[Corollary~4]{rasm07}
that the only integral surgery on $S^3$ that produces
the lens space $L(4m+3,4)$ is surgery along the 
negative torus knot $T(2,-(2m+1))$ with coefficient~$-(4m+3)$.
(The statement about the surgery coefficient is not contained
in~\cite{rasm07}, but this follows immediately from~\cite{mose71},
because other surgery coefficients lead to a different lens space or
a Seifert manifold with three multiple fibres.) Beware that
Rasmussen uses the opposite orientation convention for lens spaces.

Based on Rasmussen's result, O.~Plamenevskaya claimed in
\cite[Proposition~5.4]{plam12} that only one of the three, up to
isotopy, (positive) tight contact structures on $L(7,4)$
can be obtained via Legendrian surgery on some contact structure on~$S^3$.
This assertion, as we shall see, is erroneous. The main result of this
note is the following.

\begin{thm}
\label{thm:main}
For any pair of integers $n\geq 2$, $s\geq 1$, every
tight contact structure on the lens space $L(ns^2-s+1,s^2)$
can be obtained by a single Legendrian surgery along a suitable
Legendrian realisation of the negative torus knot $T(s,-(sn-1))$
in some contact structure on~$S^3$.
\end{thm}

\begin{rem}
\label{rem:main}
(1) On the lens space $L(ns^2-s+1,s^2)$ there are, for $s\geq 2$,
precisely $(s+1)(n-1)$ distinct tight contact structures
\emph{up to isotopy}; for $s=1$ the number is $n-1$.
(One reason why we restrict attention to these
lens spaces is that the arithmetic for determining the number
of tight structures is simple.) As we shall be realising these structures by
different contact surgery diagrams, it may not seem to be clear
how to distinguish non-isotopic but diffeomorphic structures.
However, those different contact
surgery diagrams will correspond to the same topological surgery
diagram, and this fixes the resulting manifold, so that the notion of
isotopy becomes meaningful.

(2) There is a more fundamental reason for considering only
the lens spaces $L(ns^2-s+1,s^2)$: a major share of the tight contact
structures stem from \emph{exceptional} realisations of the
torus knot $T(s,-(sn-1))$ (i.e.\ as Legendrian knots
$L\subset(S^3,\xiot)$ in an overtwisted contact structure
with $\xiot|_{S^3\setminus L}$ tight), and systematically
we can only produce them for those particular torus knots.
One can expect that the lens spaces $L(ns^2+s+1,s^2)$,
coming from surgery along the torus knots $T(s,-(sn+1))$,
may be treated in the same fashion.

(3) The maximal Thurston--Bennequin number of Legendrian
realisations of the negative torus knot $T(a,-b)$ in the
tight $(S^3,\xist)$ equals $-ab$, so the maximal topological
surgery coefficient for a Legendrian surgery is $-(ab+1)$.
It seems reasonable to conjecture that the theorem holds true
for all lens spaces $L(ab+1,a^2)$.
\end{rem}

We assume that the reader is familiar with the elements of contact
topology on the level of~\cite{geig08}. In particular, our argument depends
on the presentation of contact $3$-manifolds in terms of
contact $(\pm 1)$-surgery diagrams, see~\cite{dgs04} and
\cite[Section~6.4]{geig08}.
\section{Contact structures on lens spaces}
As explained in the introduction, lens spaces come with a natural
orientation. When we speak of a `contact structure' $\xi$ on
an oriented $3$-manifold, we always mean a positive, oriented
contact structure, i.e.\ $\xi=\ker\alpha$ with $\alpha\wedge\rmd\alpha >0$,
and $\alpha$ is given up to multiplication by a positive function.

The number of tight contact structures on lens spaces
has been determined independently by E.~Giroux~\cite{giro00}
and K.~Honda~\cite{hond00}.

\begin{thm}[Giroux, Honda]
\label{thm:GH}
On the lens space $L(p,q)$ with $p>q>0$ and $\gcd(p,q)=1$, the number
of tight contact structures is given by
\[ (a_0-1)\cdots (a_k-1),\]
where the $a_i\geq 2$ are the terms in the negative
continued fraction expansion
\[ \frac{p}{q}=a_0-\cfrac{1}{a_1-\cfrac{1}{a_2-\cdots-\cfrac{1}{a_k}}}
=:[a_0,\ldots,a_k]. \]
\end{thm}

For our family of lens spaces, this number is easy to compute.

\begin{cor}
The number of tight contact structures on $L(ns^2-s+1,s^2)$ is
$n-1$ for $s=1$, and $(s+1)(n-1)$ for $s\geq 2$.
\end{cor}

\begin{proof}
The case $s=1$ is clear. For $s\geq 2$, we claim that
\[ \frac{ns^2-s+1}{s^2}=[n,s+2,\underbrace{2,\ldots,2}_{s-2}].\]
The result then follows from Theorem~\ref{thm:GH}.

Inductively one sees that
\[ [\underbrace{2,\ldots,2}_{s-2}]=\frac{s-1}{s-2}.\]
Then
\begin{eqnarray*}
[n,s+2,\underbrace{2,\ldots,2}_{s-2}]
  & = & n-\cfrac{1}{s+2-\cfrac{1}{[2,\ldots,2]}}\\
  & = & \frac{ns^2-s+1}{s^2},
\end{eqnarray*}
as was claimed.
\end{proof}

Our aim is to find Legendrian realisations of the torus knot
$T(s,-(sn-1))$ in some contact structure on $S^3$ such that
Legendrian surgery on the knot produces a tight contact
structure on the lens space $L(ns^2-s+1,s^2)$.
This requires, first of all, that the Thurston--Bennequin invariant $\tb$
of these realisations equals $-s(sn-1)$, so that topologically
we perform a surgery with framing $-(ns^2-s+1)$. Secondly,
a necessary condition for the contact structure on the surgered manifold
to be tight is that we start with the standard tight contact structure
$\xist$ on~$S^3$, or with an exceptional realisation of $T(s,-(sn-1))$
in an overtwisted contact structure; see Remark~\ref{rem:main}~(2)
for the definition of exceptional Legendrian knots -- these are also
referred to as nonloose Legendrian knots.

We shall be representing $T(s,-(sn-1))$ as a Legendrian knot
$L$ in a contact surgery diagram of some contact structure on $S^3$.
One will then see directly that Legendrian surgery on this
knot produces a contact surgery diagram for a tight contact
structure on $L(ns^2-s+1,s^2)$. Thus, with hindsight the
contact structure on the complement of $L$ was tight.

In order to distinguish the contact structures obtained in this
fashion, we need two well-known homotopical invariants of
tangent $2$-plane fields on $3$-manifolds. The first
is the Euler class, which modulo $2$-torsion detects
homotopy over the $2$-skeleton; the second is the 
so-called $d_3$-invariant.

\begin{ex}
The Euler class suffices to settle the case $s=1$ of the
theorem. In $(S^3,\xist)$ there are $n-1$ Legendrian realisations of the
(oriented) unknot with $\tb=-n+1$, with rotation numbers in the range
\[ \rot\in\{-n+2,-n+4,\ldots, n-4,n-2\},\]
given by adding $n-2$ zigzags, distributed on the left and the
right, to the standard front projection picture of an unknot
with $\tb=-1$ and $\rot=0$. Legendrian surgery on these knots
produces the lens space $L(n,1)$. The almost complex structure
on the corresponding symplectic handlebody~$X$, which has
cohomology $H^2(X)\cong\Z$, has first Chern class equal
to $\rot$, and by \cite[Theorem~1.2]{lima97}, this
distinguishes the $n-1$ contact structures on $L(n,1)$ up to isotopy.
\end{ex}

Recall from \cite[Section~4]{gomp98} that one can associate with any
oriented tangent $2$-plane field $\eta$ on a closed, orientable $3$-manifold
$Y$ a homotopy invariant $d_3(\eta)\in\Q$, provided the Euler
class $e(\eta)$ is a torsion class. This is the homotopy obstruction
over the $3$-skeleton of $Y$ in the sense that two such $2$-plane fields
that are homotopic over the $2$-skeleton
are homotopic over $Y$ if and only if they have
the same $d_3$-invariant.

Suppose the contact manifold $(Y,\xi)$ is given in terms of
a surgery presentation $\LL=\LL_+\sqcup\LL_-\subset(S^3,\xist)$,
i.e.\ $\LL$ is a Legendrian link, and  $(Y,\xi)$
is the result of performing
contact $(\pm 1)$-surgery along the components of the
sublinks $\LL_{\pm}$.
In this situation, the $d_3$-invariant can be computed as follows,
see \cite[Corollary 3.6]{dgs04}, where $X$ is the
$4$-dimensional handlebody determined by the surgery
description, $\chi(X)$ its Euler characteristic, and
$\sigma(X)$ its signature.

\begin{prop}
\label{prop:d3}
Suppose that the Euler class $e(\xi)$ is torsion, and $\tb (L_i)\neq 0$
for each $L_i\in\LL_+$. Then
\begin{equation}
\label{eqn:d3}
d_3(\xi )=\frac{1}{4}\bigl( c^2-3\sigma (X)-2\chi (X)\bigr) +q,
\end{equation}
where $q$ denotes the number of components of $\LL_+$,
and $c\in H^2 (X)$ is the cohomology class
determined by $c(\Sigma_i)=\rot(L_i)$ for each $L_i\subset \LL$.
\end{prop}

For the computation of the rational number $c^2$, write $M$
for the linking matrix of the link $\LL$, with diagonal
entries given by the topological surgery framings. Let $\vrot$
be the vector of rotation numbers of the link components. Solve
the linear system $M\bfx=\vrot$ over~$\Q$. Then $c^2=\bfx^{\ttt}M\bfx$.
\section{The lens space $L(7,4)$}
\label{section:n=2,s=2}
In order to see where the error occurs
in \cite{plam12}, we begin with Plamenevskaya's example $L(7,4)$,
i.e.\ the case $n=2$, $s=2$. This lens space admits three
tight contact structures.

According to Rasmussen's result cited in the introduction,
the only way to obtain $L(7,4)$
by an integral surgery on $S^3$ is a $(-7)$-surgery along
a left-handed trefoil knot $T(2,-3)$. We therefore need to look
for Legendrian realisations of the left-handed trefoil
with $\tb=-6$ either in $(S^3,\xist)$ or as an exceptional
knot in an overtwisted contact structure on~$S^3$.

By \cite[Section~4.1]{etho01}, $-6$ is the maximal
Thurston--Bennequin number for Legendrian realisations of the
left-handed trefoil knot in $(S^3,\xist)$, and there
are precisely two realisations (as oriented Legendrian knots)
with this maximal $\tb$,
distinguished by their rotation numbers $\rot=\pm 1$.
They are shown in Figure~\ref{figure:lh-trefoil-pm}.

\begin{figure}[h]
\labellist
\small\hair 2pt
\endlabellist
\centering
\includegraphics[scale=1.4]{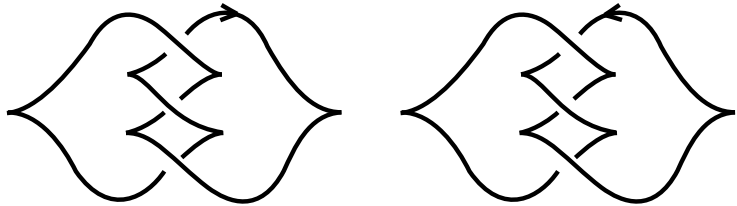}
  \caption{The left-handed trefoils in $(S^3,\xist)$ with $\tb=-6$.}
  \label{figure:lh-trefoil-pm}
\end{figure}

An exceptional realisation $L\subset S^3$ of the left-handed trefoil is shown
in Figure~\ref{figure:exc-lh-trefoil}.
We shall check in a moment that $\tb(L)=-6$; Legendrian surgery along
$L$ then produces~$L(7,4)$.

\begin{figure}[h]
\labellist
\small\hair 2pt
\pinlabel $L$ [l] at 305 648
\pinlabel $+1$ [l] at 305 658
\pinlabel $+1$ [l] at 305 665
\pinlabel $-1$ [l] at 305 685
\pinlabel $-1$ [l] at 305 708
\pinlabel $-1$ [l] at 294 732
\endlabellist
\centering
\includegraphics[scale=1.5]{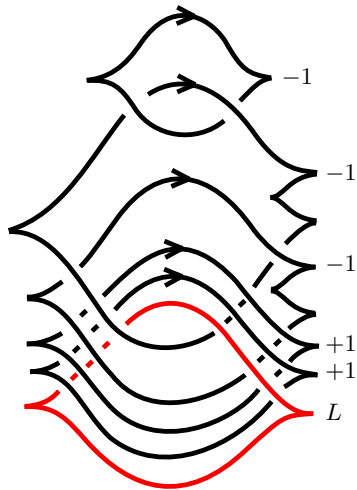}
  \caption{An exceptional left-handed trefoil $L$ with $\tb=-6$.}
  \label{figure:exc-lh-trefoil}
\end{figure}

The Kirby moves to verify the topological
part of this statement are shown in Figures \ref{figure:exc-lh-trefoil-Kirby}
and~\ref{figure:torus-knot-slides}. The latter shows how to separate
$L$ from the surgery link by $1+2$ handle slides (where in the first
step we slide both strands simultaneously), turning it
into a left-handed trefoil in $S^3$. The two-component surgery link
does indeed represent the $3$-sphere, since surgery along
the $0$-framed meridian cancels the $(-2)$-surgery;
for instance, one can use a slam-dunk~\cite[Figure~5.30]{gost99}.

\begin{figure}[h]
\labellist
\small\hair 2pt
\pinlabel $-2$ [br] at 158 782
\pinlabel $-3$ [tr] at 162 755
\pinlabel $-3$ [bl] at 172 752
\pinlabel $-1$ at 204 762
\pinlabel $0$ [l] at 239 762
\pinlabel $0$ [tl] at 240 747
\pinlabel $-2$ [br] at 275 780
\pinlabel $-2$ [tr] at 284 750
\pinlabel $-2$ [bl] at 293 755
\pinlabel $1$ [t] at 323 760
\pinlabel $1$ [l] at 354 762
\pinlabel $1$ [tl] at 359 751
\pinlabel $-2$ [br] at 149 709
\pinlabel $-2$ [tr] at 159 678
\pinlabel $-2$ [bl] at 167 683
\pinlabel $-1$ [b] at 205 698
\pinlabel $-2$ [br] at 244 710
\pinlabel $-1$ [tr] at 252 679
\pinlabel $-1$ [bl] at 261 684
\pinlabel $+1$ at 285 691
\pinlabel $-2$ [br] at 328 710
\pinlabel $-1$ [tr] at 337 677
\pinlabel $-1$ [t] at 373 681
\pinlabel $+1$ at 373 693
\pinlabel $-2$ [br] at 191 638
\pinlabel $0$ [tr] at 199 607
\pinlabel $+2$ at 233 619
\pinlabel $-2$ [br] at 288 632
\pinlabel $0$ [t] at 310 617
\endlabellist
\centering
\includegraphics[scale=1.3]{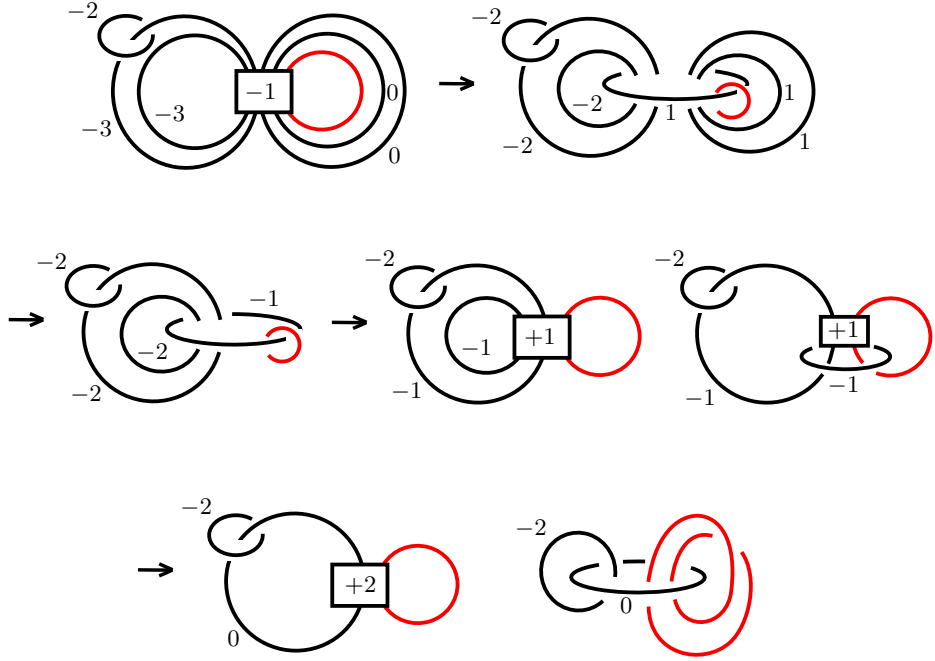}
  \caption{Kirby moves for Figure~\ref{figure:exc-lh-trefoil}.}
  \label{figure:exc-lh-trefoil-Kirby}
\end{figure}

\begin{figure}[h]
\labellist
\small\hair 2pt
\pinlabel $-2$ [br] at 242 767
\pinlabel $0$ [t] at 264 752
\pinlabel $-2$ [l] at 196 710
\pinlabel $-2$ at 217 714
\pinlabel $0$ [bl] at 241 706
\pinlabel $0$ [br] at 282 716
\pinlabel $-2$ [bl] at 320 718
\pinlabel $-2$ at 334 702
\pinlabel $0$ [bl] at 196 649
\pinlabel $-2$ [bl] at 235 661
\pinlabel $-2$ at 247 645
\pinlabel $2\text{ times}$ [r] at 200 636
\pinlabel $-2$ at 324 647
\endlabellist
\centering
\includegraphics[scale=1.6]{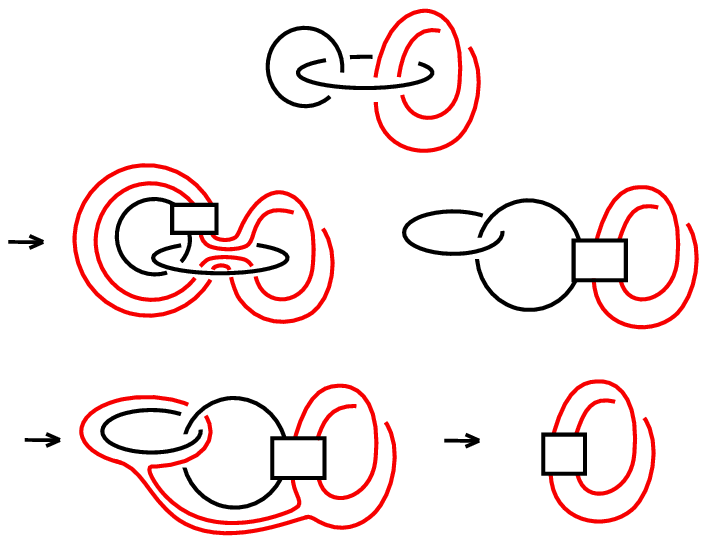}
  \caption{The final handle slides.}
  \label{figure:torus-knot-slides}
\end{figure}

To see that $L$ is exceptional, observe that -- by the cancellation
lemma~\cite{dige04}, cf.\ \cite[Proposition 6.4.5]{geig08} --
contact $(-1)$-surgery
along $L$ in Figure~\ref{figure:exc-lh-trefoil} cancels one of the contact
$(+1)$-surgeries. Then the remaining surgery diagram
contains only a single contact $(+1)$-surgery along a standard
Legendrian unknot, which by \cite{dgs04} produces the unique tight
(and Stein fillable) contact structure on $S^1\times S^2$,
and the further contact $(-1)$-surgeries then produce
a Stein fillable and hence tight contact structure.

On the other hand, the contact structure on $S^3$ given by
the surgery diagram in Figure~\ref{figure:exc-lh-trefoil}
(without any surgery along~$L$) is overtwisted,
since its $d_3$-invariant is $3/2$ (recall that $d_3(\xist)=-1/2$).
Indeed, the linking matrix (ordering the knots from bottom to top) is
\[ M=\begin{pmatrix}
0  & -1 & -1 & -1 & 0\\
-1 & 0  & -1 & -1 & 0\\
-1 & -1 & -3 & -1 & 0\\
-1 & -1 & -1 & -3 & -1\\
0  & 0  & 0  & -1 & -2
\end{pmatrix}\]
with signature $\sigma=-1$ (instead of computing $\sigma$ from the
matrix one can see this from the Kirby moves in
Figure~\ref{figure:exc-lh-trefoil-Kirby}
by keeping track of the blow-ups and blow-downs).
Since we are adding five $2$-handles to the $4$-ball, the Euler
characteristic of the handlebody is $\chi=6$. The vector
of rotation numbers is $\vrot=(0,0,1,1,0)^{\ttt}$, and the solution of
$M\bfx=\vrot$ is $\bfx=(-7,-7,3,4,-2)^{\ttt}$. This gives
$c^2=\bfx^{\ttt}M\bfx=7$, and hence $d_3=3/2$ with formula~(\ref{eqn:d3}).

We still need to check that $\tb(L)=-6$. For this we may use
the formula from \cite[Lemma~6.6]{loss09}, cf.~\cite[Lemma~3.1]{geon15}
and~\cite{duke16}.
Consider the extended linking matrix
\[ M_0=\begin{pmatrix}
0  & -1 & -1 & -1 & -1 & 0\\
-1 & 0  & -1 & -1 & -1 & 0\\
-1 & -1 & 0  & -1 & -1 & 0\\
-1 & -1 & -1 & -3 & -1 & 0\\
-1 & -1 & -1 & -1 & -3 & -1\\
0  & 0  & 0  & 0  & -1 & -2
\end{pmatrix},\]
which now includes $L$ as the first link component, with the first
diagonal entry set to zero. Write $\tb_0$ for the Thurston--Bennequin
invariant of $L$ as a knot in the unsurgered copy of $S^3$, i.e.\
here $\tb_0=-1$. Then, in the surgered manifold (which here is
another copy of~$S^3$) one has
\[ \tb(L)=\tb_0+\frac{\det M_0}{\det M}=-1+\frac{5}{-1}=-6.\]

\begin{rem}
\label{rem:tb-2}
Alternatively, one can determine $\tb(L)$ by keeping track
of the framing of $L$ during the Kirby moves. Start with a
Legendrian push-off $L'$ of $L$ in the original diagram, which has
linking $-1$ with~$L$. In the last diagram of
Figure~\ref{figure:exc-lh-trefoil-Kirby}, we then have linking
number $\lk(L,L')=2$ (as knots in the unsurgered~$S^3$).
After the handle slides in Figure~\ref{figure:torus-knot-slides},
the parallel knot $L'$ will likewise pass twice through the $(-2)$-box,
so two strands of $L'$ will each receive a $(-2)$-twisting relative to two
strands of~$L$, resulting in $\tb(L)=2-2^3=-6$.
\end{rem}

\begin{prop}
\label{prop:L(7,4)}
Legendrian surgery along the three left-handed trefoil knots
in Figures~\ref{figure:lh-trefoil-pm}
and~\ref{figure:exc-lh-trefoil} produces the three tight contact structures
on $L(7,4)$.
\end{prop}

\begin{rem}
A word of clarification is in order. The result of a surgery
along a knot does not depend on the orientation of the knot.
In what sense, then, can the two knots in Figure~\ref{figure:lh-trefoil-pm}
be said to correspond to two non-isotopic contact structures on a given
copy of the lens space $L(7,4)$?

In an integral surgery diagram, read as
a Kirby diagram for a $4$-dimensional handlebody~$X$,
a choice of orientation on a knot $K$ amounts to a choice
of positive generator in the corresponding $\Z$-summand of $H_2(X)$,
represented by an oriented Seifert surface for $K$ glued
with the core disc in the handle.

The left-handed trefoil, like all torus knots, is a reversible
knot, i.e.\ it is isotopic to itself with reversed orientation.
This means that the two oriented knots in Figure~\ref{figure:lh-trefoil-pm}
are topologically isotopic. The time-$1$ map of this isotopy
gives us an identification of the two handlebodies such that
either knot corresponds to the positive generator of $H_2(X)$.
It is with respect to this identification that we compare
the resulting contact structures; likewise for the
contact structure coming from Figure~\ref{figure:exc-lh-trefoil}.

The same comments apply to the general case discussed in the
subsequent sections. In particular, once an orientation of
the torus knot has been fixed, we can use the Euler class
to distinguish contact structures on the lens space resulting from
surgery, without any need to
consider the action of diffeomorphisms on the second
cohomology group.
\end{rem}

\begin{proof}[Proof of Proposition~\ref{prop:L(7,4)}]
First of all, observe that Legendrian surgery along the
examples in Figure~\ref{figure:lh-trefoil-pm} produces
Stein fillable and hence tight contact structures on $L(7,4)$.
Tightness of the contact structure obtained by surgery along the
exceptional left-handed trefoil in Figure~\ref{figure:exc-lh-trefoil}
was explained above; this was our argument for seeing that the Legendrian
knot $L$ is indeed exceptional.

For the examples in Figure~\ref{figure:lh-trefoil-pm} we compute $d_3=-2/7$
from $M=(-7)$, $\chi=2$, $\sigma=-1$,
and $c^2=-1/7$. By \cite[Proposition~2.3]{gomp98},
the first Chern class of the Stein surface $(X,J)$ described by
the respective diagram evaluates to $\rot=\pm 1$ on the positive
generator of $H_2(X)$. The Euler class of the respective contact structure
induced on the boundary $\partial X= L(7,4)$ then equals
$\pm 1\in H^2(L(7,4))=\Z_7$, so the two contact structures
are non-isotopic.

For the example in Figure~\ref{figure:exc-lh-trefoil},
Legendrian surgery cancels
one of the contact $(+1)$-surgeries. The remaining diagram has linking
matrix (ordering the knots from bottom to top)
\[ M=\left(\begin{array}{cccc}
0  & -1 & -1 & 0\\
-1 & -3 & -1 & 0\\
-1 & -1 & -3 & -1\\
0  & 0  & -1 & -2
\end{array}\right)\]
with signature $\sigma=-2$. The Euler characteristic of the
handlebody is $\chi=5$, and the vector of rotation numbers
is $\vrot=(0,1,1,0)^{\ttt}$. The solution of $M\bfx=\vrot$
is given by $\bfx=(-1,0,0,0)^{\ttt}$, hence $c^2=\bfx^{\ttt}M\bfx=0$,
and $d_3=(0+6-10)/4+1=0$, which distinguishes this contact
structure from the other two.
\end{proof}

\begin{rem}
For Legendrian surgery diagrams of the three tight structures
on $L(7,4)$ involving two-component links see
\cite[Figure~6]{plam12}. The error in \cite{plam12}
occurs in the computation of the $d_3$-invariant for
the contact structure coming from the examples in
Figure~\ref{figure:lh-trefoil-pm};
Plamenevskaya obtains $d_3=0$ and then argues, correctly,
that the two structures with $d_3=-2/7$ cannot come from an exceptional
trefoil.
\end{rem}

\section{The lens spaces $L(4m+3,4)$}
\label{section:s=2}
As a first generalisation, we now prove Theorem~\ref{thm:main}
for $s=2$ but arbitrary~$n$. Here we shall distinguish
the contact structures by their Euler classes.
For notational convenience, we set $m=n-1$, so instead
of $L(4n-1,4)$ we consider $L(4m+3,4)$ with $m\geq 1$.

The lens space $L(4m+3,4)$ is obtained by $-(4m+3)$-surgery
on the torus knot $T(2,-(2m+1))$, so we need to find
Legendrian realisations of this torus knot with $\tb=-(4m+2)$.

The following result is due to Etnyre--Honda~\cite[Section~4.1]{etho01},
where the reader can also find explicit front projection diagrams
of the knots in question.

\begin{prop}[Etnyre--Honda]
The maximal Thurston--Bennequin invariant of Legendrian
realisations of $T(2,-(2m+1))$ in $(S^3,\xist)$ is $-(4m+2)$.
Up to Legendrian isotopy, there are $2m$ realisations with this maximal~$\tb$,
with rotation number in the range
\[ \rot\in\{-2m+1,-2m+3,\ldots,2m-3,2m-1\}. \]
\end{prop}

Legendrian surgery on these knots
yields $2m$ tight contact structures on the lens space
$L(4m+3,4)$, distinguished
by their Euler class in $H^2(L(4m+3,4))=\Z_{4m+3}$, which, as above,
is given by the reduction of the rotation number modulo $4m+3$.

The remaining $m$ tight contact structures
on $L(4m+3,4)$ have to come from exceptional realisations of the
torus knot $T(2,-(2m+1))$ in some
overtwisted contact structure on $S^3$.

\begin{prop}
\label{prop:exc-2-n}
For $(k,l)\in\N_0\times \N_0$ with $k+l=m-1$, the Legendrian
knot $L$ shown in the contact surgery diagram of
Figure~\ref{figure:neg-torus-knot-2} is an exceptional realisation
of $T(2,-(2m+1))$ in $S^3$ with $\tb=-(4m+2)$.
\end{prop}

\begin{figure}[h]
\labellist
\small\hair 2pt
\pinlabel $-1$ [bl] at 279 776
\pinlabel $k$ [r] at 239 754
\pinlabel $l$ [l] at 301 754
\pinlabel $-1$ [l] at 305 709
\pinlabel $-1$ [l] at 305 685
\pinlabel $+1$ [l] at 305 665
\pinlabel $+1$ [l] at 305 658
\pinlabel $L$ [l] at 305 648
\endlabellist
\centering
\includegraphics[scale=1.5]{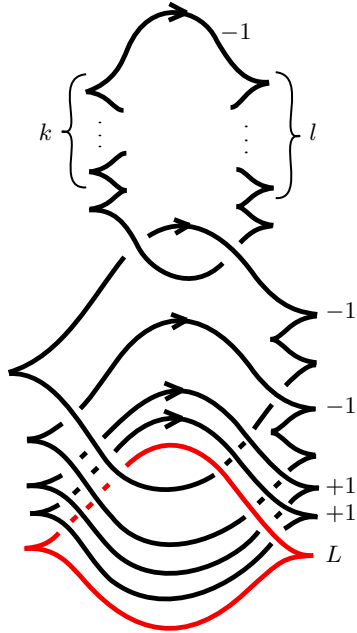}
  \caption{Exceptional realisations of $T(2,-(2m+1))$, $m=k+l+1$.}
  \label{figure:neg-torus-knot-2}
\end{figure}

\begin{rem}
Exceptional realisations of the torus knots $T(2,-(2m+1))$
have previously been described in \cite{loss09}. The Kirby moves
in Figure~\ref{figure:exc-lh-trefoil-Kirby} are those
of \cite[Figure~18]{loss09}. However, the purported
exceptional realisations in \cite[Figure~17]{loss09} do not actually
correspond to this Kirby diagram.
\end{rem}

\begin{proof}[Proof of Proposition~\ref{prop:exc-2-n}]
The Kirby moves in Figures \ref{figure:exc-lh-trefoil-Kirby}
and~\ref{figure:torus-knot-slides} once again confirm
the topological part of the statement. Simply replace the
$(-2)$-framing in the first diagram of
Figure~\ref{figure:exc-lh-trefoil-Kirby} by $-(m+1)=-n$,
which remains unchanged throughout the moves,
and instead of the $1+2$ slides in Figure~\ref{figure:torus-knot-slides}
perform $1+n$ slides. Likewise, the argument that surgery
along $L$ in Figure~\ref{figure:neg-torus-knot-2} produces a tight
contact structure is as before.

For the computation of $\tb(L)$ in the
surgered $S^3$, observe that only the last diagonal entry $-2$
in the matrices $M$ and $M_0$ (before Proposition~\ref{prop:L(7,4)})
needs to be replaced by $-(m+1)$.
By expanding the respective determinants along the last row,
one finds that $\det M=-1$ remains unchanged, and $\det M_0=-3+4(m+1)=
4m+1$. This gives $\tb(L)=-(4m+2)$, as desired.

Next we compute the $d_3$-invariant of the surgered~$S^3$. As before we have
$\sigma=-1$ and $\chi=6$. We have
\[ \vrot =(0,0,1,1,l-k)^{\ttt},\]
and the solution of $M\bfx=\vrot$ is given by
\[ \bfx=(-6k-2l-7,-6k-2l-7,3k+l+3,3k+l+4,-2)^{\ttt}.\]
It follows that $d_3=2k+3/2$, so the contact structure
is overtwisted, and $L$ is exceptional.
\end{proof}

The following proposition shows that Legendrian surgery along these
exceptional knots produces the required $m$ tight structures
on $L(4m+3,4)$, distinct from the $2m$ we found previously.

\begin{prop}
Legendrian surgery along the knot $L$ shown in
Figure~\ref{figure:neg-torus-knot-2} (for the various choices
of $k,l$) produces $m$ tight contact
structures on $L(4m+3,4)$ with Euler class in $H^2(L(4m+3,4))=
\Z_{4m+3}$ in the range
\[ -2m+2,-2m+6, \ldots,2m-6,2m-2\mod 4m+3.\]
\end{prop}

\begin{proof}
We compute the Euler class of the contact structures
on $L(4m+3,4)$ obtained via Legendrian surgery along these
exceptional knots. To this end, we need to compute the
rotation number of $L$ as a knot in the overtwisted $S^3$
obtained by surgery along the link in
Figure~\ref{figure:neg-torus-knot-2}. According to
\cite[Lemma~6.6]{loss09}, cf.\ \cite[Lemma~3.1]{geon15},
this rotation number is given by
\[ \rot(L)=\rot_0-\langle\vrot,M^{-1}\vlk\rangle,\]
where $\rot_0$ denotes the rotation number of $L$
in the unsurgered copy of $S^3$ (here $\rot_0=0$),
and $\vlk$ is the vector of linking numbers of $L$ with
the components of the surgery link. We give $L$ the clockwise
orientation in the diagram, then
\[ \vlk=(-1,-1,-1,-1,0)^{\ttt}\]
and
\[ M^{-1}\vlk=(-4m-2,-4m-2,2m+1,2m+2,-2)^{\ttt}.\]
This yields
\[ \rot(L)=-(4m+3)+2(l-k)\equiv 2(l-k)\mod 4m+3.\]
Hence, modulo $4m+3$, the rotation number can take on $m$
distinct values in the range
\[ \rot\in\{-2m+2,-2m+6, \ldots,2m-6,2m-2\}.\qedhere\]
\end{proof}

Here is an alternative way to compute the Euler class of the
contact structure on $L(4m+3,4)$ obtained by surgery on
the link in Figure~\ref{figure:neg-torus-knot-2} (including
a contact $(-1)$-surgery along~$L$), which is more direct
than computing $\rot(L)$ in the surgered~$S^3$.

Write $X$ for the $4$-dimensional handlebody described by this
diagram, so that $\partial X=L(4m+3,4)$. The first
homology group $H_1(\partial X)=\Z_{4m+3}$ is generated by the
classes of the meridians $[\mu_L],[\mu_1],\ldots,[\mu_5]$, and the
relations are given by the linking matrix, cf.\ \cite{dgs04}.
For ease of notation, we change all signs in that matrix:
\[ \begin{pmatrix}
2 & 1 & 1 & 1 & 1 & 0\\
1 & 0 & 1 & 1 & 1 & 0\\
1 & 1 & 0 & 1 & 1 & 0\\
1 & 1 & 1 & 3 & 1 & 0\\
1 & 1 & 1 & 1 & 3 & 1\\
0 & 0 & 0 & 0 & 1 & m+1
\end{pmatrix} \cdot
\begin{pmatrix}
[\mu_L]\\
[\mu_1]\\
[\mu_2]\\
[\mu_3]\\
[\mu_4]\\
[\mu_5]
\end{pmatrix} =
\begin{pmatrix}
0\\
0\\
0\\
0\\
0\\
0
\end{pmatrix}.\]
This yields the relations
\[ (4m+3)[\mu_4]=0,\;\;\;
-[\mu_L]=[\mu_1]=[\mu_2]=2[\mu_4],\;\;\; [\mu_3]=-[\mu_4],\;\;\;
[\mu_5]=-4[\mu_4],\]
so that indeed $H_1(\partial X)=\Z_{4m+3}$, generated by
the class~$[\mu_4]$. The class $[\mu_L]$ can likewise be
taken as a generator, since $-2(m+1)[\mu_L]=(4m+4)[\mu_4]=[\mu_4]$.

Now, as discussed in \cite{dgs04}, the Poincar\'e dual
of the Euler class $e(\xi)$ of the contact structure on $\partial X$
is given by the vector of rotation numbers of the link components,
expressed in terms of the classes of meridians, that is,
\[ e(\xi)=[\mu_3]+[\mu_4]+(l-k)[\mu_5]=2(l-k)[\mu_L],\]
which confirms the calculation in the foregoing proof.
\section{The lens spaces $L(ns^2-s+1,s^2)$}
Finally, we deal with the general case $n\geq 2$, $s\geq 2$
of Theorem~\ref{thm:main}. The relevant result from \cite{etho01} can
now be phrased as follows.

\begin{prop}[Etnyre--Honda]
\label{prop:EH-s}
The maximal Thurston--Bennequin invariant of Legendrian
realisations of $T(s,-(sn-1))$ in $(S^3,\xist)$ is $-s(sn-1)$.
Up to Legendrian isotopy, there are $2(n-1)$ realisations with this
maximal~$\tb$, with rotation number in the range
\[ \{ -(n-1)s+1,-(n-3)s\pm 1,\ldots,(n-3)s\pm 1,(n-1)s-1\}.\]
\end{prop}

As in the case $s=2$, Legendrian surgery along these knots
gives us $2(n-1)$ tight contact structures on $L(ns^2-s+1,s^2)$.
It remains to find $(s-1)(n-1)$ exceptional realisations
of $T(s,-(sn-1))$ in $S^3$ that will gives us the remaining tight
structures on the lens space.

\begin{prop}
\label{prop:exc-s-n}
For $(k,l)\in\N_0\times\N_0$ with $k+l=n-2$ and $(p,q)\in\N_0\times\N$
(sic!) with $p+q=s-1$, the Legendrian knot $L$ shown in
Figure~\ref{figure:neg-torus-knot-gen} is an exceptional realisation of
$T(s,-(sn-1))$ in $S^3$ with $\tb=-s(sn-1)$.
\end{prop}

\begin{figure}[h]
\labellist
\small\hair 2pt
\pinlabel $-1$ [bl] at 280 805
\pinlabel $k$ [r] at 238 783
\pinlabel $l$ [l] at 300 783
\pinlabel $-1$ [bl] at 296 739
\pinlabel $p$ [r] at 224 722
\pinlabel $q$ [l] at 322 710
\pinlabel $-1$ [bl] at 285 716
\pinlabel $-1$ [bl] at 277 704
\pinlabel $+1$ [l] at 304 665
\pinlabel $+1$ [l] at 304 658
\pinlabel $L$ [l] at 304 648
\pinlabel $s-1\text{ knots}$ [r] at 228 685
\endlabellist
\centering
\includegraphics[scale=1.5]{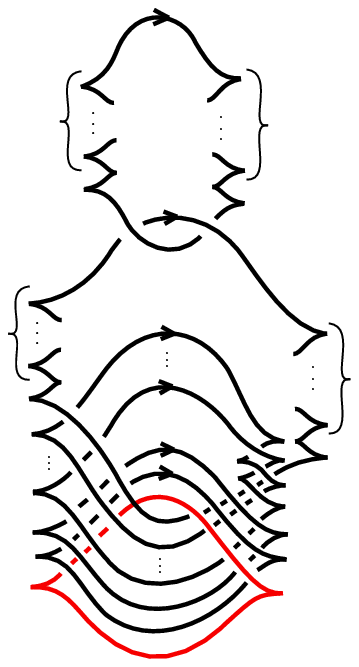}
  \caption{Exceptional realisations of $T(s,-(sn-1))$.}
  \label{figure:neg-torus-knot-gen}
\end{figure}

\begin{proof}
For the topological aspect of the statement, see the
Kirby moves in Figure~\ref{figure:neg-torus-knot-Kirby}.
The effect of the final $1+n$ handle slides is now
shown in Figure~\ref{figure:slam}. Before, we explained that
the cancellation of the two surgeries (after the handle slides)
can be interpreted as a slam-dunk of the $0$-framed meridian,
which turns the $(-n)$-framed unknot into an unknot with
surgery framing $-n-1/0=\infty$. Alternatively, we can
actually interpret the complete move (handle slides and
cancellation of the surgeries) as a slam-dunk of the
$(-n)$-framed meridian to the $0$-framed unknot. This turns the
latter into an unknot with framing $0-1/(-n)=1/n$, and leaves
$L$ unchanged. The $(1/n)$-surgery along the unknot is
equivalent to removing a tubular neighbourhood of the unknot,
twisting the neck $-n$ times, and then regluing the solid
torus with the identity map. This results in the diagram
on the right of Figure~\ref{figure:slam}.

\begin{figure}[h]
\labellist
\small\hair 2pt
\pinlabel $-3$ [br] at 173 781
\pinlabel $-3$ [tl] at 181 776
\pinlabel $-1$ at 170 762
\pinlabel $-1$ at 211 762
\pinlabel $-n$ [bl] at 256 789
\pinlabel $-s-1$ [br] at 221 785
\pinlabel $0$ [l] at 251 762
\pinlabel $0$ [r] at 246 762
\pinlabel $-2$ [br] at 292 782
\pinlabel $-2$ [tl] at 298 774
\pinlabel $-1$ at 288 764
\pinlabel $1$ [b] at 325 762
\pinlabel $1$ [r] at 360 763
\pinlabel $1$ [l] at 368 763
\pinlabel $-s$ [br] at 337 784
\pinlabel $-n$ [bl] at 375 790
\pinlabel $-2$ [br] at 240 712
\pinlabel $-2$ [tl] at 249 706
\pinlabel $-1$ at 237 695
\pinlabel $-1$ [b] at 275 693
\pinlabel $-s$ [br] at 288 712
\pinlabel $-n$ [bl] at 323 713
\pinlabel $-1$ [br] at 210 646
\pinlabel $-1$ [tl] at 218 639
\pinlabel $-1$ at 207 628
\pinlabel $+1$ at 240 627
\pinlabel $-s+1$ [br] at 255 647
\pinlabel $-n$ [bl] at 281 647
\pinlabel $-1$ [r] at 302 624
\pinlabel $-1$ [r] at 302 614
\pinlabel $+1$ at 310 633
\pinlabel $-s+1$ [br] at 318 650
\pinlabel $-n$ [bl] at 351 650
\pinlabel $-n$ [br] at 197 576
\pinlabel $0$ [tr] at 203 550
\pinlabel $+s$ at 243 558
\pinlabel $-n$ [br] at 288 569
\pinlabel $0$ [t] at 309 554
\endlabellist
\centering
\includegraphics[scale=1.4]{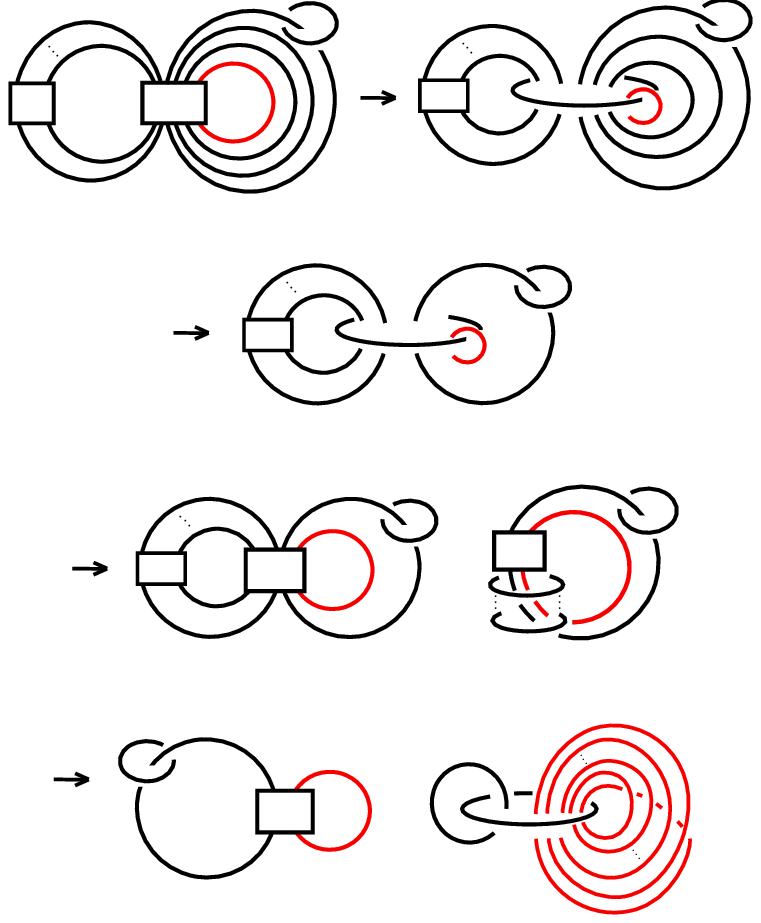}
  \caption{Kirby moves for Figure~\ref{figure:neg-torus-knot-gen}.}
  \label{figure:neg-torus-knot-Kirby}
\end{figure}

\begin{figure}[h]
\labellist
\small\hair 2pt
\pinlabel $-n$ [br] at 164 764
\pinlabel $0$ [t] at 183 747
\pinlabel $-n$ at 285 749
\endlabellist
\centering
\includegraphics[scale=1.9]{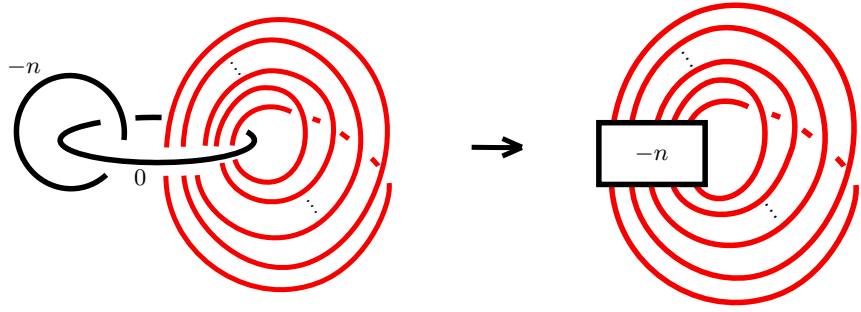}
  \caption{The final handle slides as a slam-dunk.}
  \label{figure:slam}
\end{figure}

The linking matrix, which we need for the computation of various
invariants, is now

\[ M=\left(\begin{array}{cccccccc}
0  & -1 & \makebox[0pt][l]{$\smash{\overbrace{
         \phantom{\begin{matrix}-1&-1&\ldots&-1\end{matrix}}}^{\text{$s-1$}}}$}
          -1 & -1 & \ldots & -1 & -1 & 0\\
-1 & 0  & -1 & -1 & \ldots & -1 & -1 & 0\\
-1 & -1 & -3 & -2 & \ldots & -2 & -1 & 0\\
-1 & -1 & -2 & -3 & \ldots & -2 & -1 & 0\\
\vdots & \vdots & \vdots & \vdots & \ddots & \vdots & \vdots & \vdots\\
-1 & -1 & -2 & -2 & \ldots & -3 & -1 & 0\\
-1 & -1 & -1 & -1 & \cdots & -1 & -s-1 & -1\\
0  & 0  & 0  & 0  & \cdots & 0  & -1   & -n
\end{array}\right).\]

In order to verify that $L$ is exceptional, we compute the $d_3$-invariant
of the contact structure on $S^3$ described by the surgery diagram
in Figure~\ref{figure:neg-torus-knot-gen}. The number of $2$-handles in
this Kirby diagram is $s+3$, so the Euler characteristic of the handlebody
is $\chi=4+s$. The signature is $\sigma=-1-(s-2)=1-s$ since, compared with
with the diagram in Figure~\ref{figure:neg-torus-knot-2}, we have
$s-2$ additional $(-1)$-framed unknots from the blow-downs
shown in Figure~\ref{figure:neg-torus-knot-Kirby}.

The vector of rotation numbers is
\[ \vrot=(0,0,\underbrace{1,\ldots,1}_{s-1},q-p,l-k)^{\ttt},\]
and the solution $\bfx$ of $M\bfx=\vrot$ is given by
\[ \bigl(-1-su, -1-su,
\underbrace{u,\ldots,u}_{s-1},u+1,-2q\bigr)^{\ttt},\]
where $u:=k-l+2qn-1$. It follows that
\[ c^2=\bfx^{\ttt}M\bfx=4nq^2+4q(k-l)-s+1,\]
and
\[ d_3=nq^2+q(k-l)-\frac{1}{2}.\]
For $q\geq 1$, as assumed in the proposition, we have $d_3>-1/2=d_3(\xist)$,
so this diagram does indeed define an overtwisted contact structure
on~$S^3$.

For the computation of $\tb(L)$ in the surgered $S^3$, one determines
the determinants of $M$ and the extended matrix $M_0$
by a simple row reduction. One finds $\det M=(-1)^{s-1}$ and
$\det M_0=(-1)^{s-1}\bigl(1-s(sn-1)\bigr)$. The formula
from Section~\ref{section:n=2,s=2} then yields $\tb(L)=-s(sn-1)$.
\end{proof}

\begin{rem}
Analogous to Remark~\ref{rem:tb-2}, one can alternatively compute
$\tb (L)$ from the effects of the Kirby moves on the
framing given by a Legendrian push-off $L'$ of $L$ in the original
diagram. After the moves in Figure~\ref{figure:neg-torus-knot-Kirby},
we have $\lk(L,L')=s$ (as knots in the unsurgered~$S^3$).
After the slam-dunk in Figure~\ref{figure:slam},
$s$ strands of the parallel knot $L'$ receive a $(-n)$-twist
relative to each of the $s$ strands of~$L$, so we arrive at
$\tb(L)=s-ns^2=-s(sn-1)$.
\end{rem}

The final lemma tells us that the tight contact structures on
$L(ns^2-s+1,s^2)$ obtained by Legendrian surgery on these
$(s-1)(n-1)$ Legendrian realisations of $T(s,-(sn-1))$ in $S^3$
can be distinguished from one another -- and from the ones coming
from Proposition~\ref{prop:EH-s} -- by their Euler class.
This completes the proof of Theorem~\ref{thm:main}.

\begin{lem}
\label{lem:e}
The Euler class of the tight contact structure $\xi=\xi_{k,l,p,q}$
on the lens space $L(ns^2-s+1,s^2)$
obtained by surgery on the link in Figure~\ref{figure:neg-torus-knot-gen},
including a contact $(-1)$-surgery along~$L$, is
\[ e(\xi)=(p-q+1)ns+(l-k)s\mod ns^2-s+1.\]
\end{lem}

\begin{rem}
Notice that $l-k$ takes values in the range
\[ l-k\in\{-n+2,-n+4,\ldots, n-4,n-2\};\]
for $p-q+1$ the range is
\[ \{ -s+2, -s+4,\ldots,s-4,s-2\}.\]
So the first summand in the expression $e(k,l,p,q)$ for $e(\xi)$ varies in
steps of size $2ns$, whereas the second summand ranges between
$\pm(n-2)s$. This means that there are no duplications in
this list of Euler numbers, at least before reducing modulo
$ns^2-s+1$.

In $\Z$ we have
\[ e_{\min}:=-ns^2+(n+2)s\leq e(k,l,p,q)\leq
ns^2-(n+2)s=:e_{\max},\]
so $e_{\max}-e_{\min}<2(ns^2-s+1)$. When we bring the negative $e(k,l,p,q)$
modulo $ns^2-s+1$ into the range $(0,ns^2-s+1)$, they are congruent to
$1$ modulo~$s$, whereas the positive $e(k,l,p,q)$ are divisible by~$s$,
so there are no duplications even modulo $ns^2-s+1$. Moreover, we have
\[ e_{\min}+ns^2-s+1=(n+1)s+1>(n-1)s-1\]
and
\[ -(n-1)s+1+ns^2-s+1=ns^2-ns+2>e_{\max},\]
which implies that there are also no duplications with the
Euler numbers (modulo $ns^2-s+1$) coming from
Proposition~\ref{prop:EH-s}.
\end{rem}

\begin{proof}[Proof of Lemma~\ref{lem:e}]
We label the meridional classes corresponding to the knots in
Figure~\ref{figure:neg-torus-knot-gen} from bottom to top as
\[ [\mu_L], [\mu_1], [\mu_2], [\nu_1],\ldots, [\nu_{s-1}],
[\alpha_s],[\beta_n].\]
We then compute the relations between these generators from the
matrix $M$ as in Section~\ref{section:s=2}, to obtain
\[ [\mu_1]=-[\mu_L],\;\;\;
[\mu_2]=-[\mu_L],\;\;\;
[\nu_1]=\ldots=[\nu_{s-1}]=(1-ns)[\mu_L],\]
\[ [\alpha_s]=-ns[\mu_L],\;\;\;
[\beta_n]=s[\mu_L],\;\;\;
(ns^2-s+1)[\mu_L]=0.\]
From the vector $\vrot$ we then compute
\begin{eqnarray*}
e(\xi) & = & [\nu_1]+\cdots+[\nu_{s-1}]+(q-p)[\alpha_s]+(l-k)[\beta_n]\\
       & = & \bigl((p-q+1)ns+(l-k)s\bigr)\,[\mu_L],
\end{eqnarray*}
as claimed.
\end{proof}

\begin{rem}
In the proofs of Propositions \ref{prop:exc-2-n} and~\ref{prop:exc-s-n}
we used the $d_3$-invariant to show that $L$ lives in an
overtwisted contact structure on~$S^3$, and hence $L$ is exceptional,
since Legendrian surgery along $L$ produces a tight contact structure.
The information that $L$ is exceptional was not actually necessary for
proving Theorem~\ref{thm:main}. Rather, the converse is true:
the fact that the resulting tight contact structure on
$L(ns^2-s+1,s^2)$ is different from any structure obtained from
a Legendrian realisation of $T(s,-(sn-1))$ in $(S^3,\xist)$
gives an alternative criterion for establishing the
exceptional character of~$L$.
\end{rem}
\begin{ack}
The research for this paper was done during an enjoyable
\emph{Research in Pairs} stay at the Mathematisches Forschungsinstitut
Oberwolfach, 27 March -- 9 April 2016. We thank the Forschungsinstitut for
its support, and its efficient and friendly staff for creating, once
again, an inspiring environment.

S.O.\ is partially supported by the grants H\"UBAP FBB-2016-9429 and
T\"UB\.ITAK \#115F519.
\end{ack}

\end{document}